\documentclass[12pt, twoside, leqno]{article}
\usepackage{amsfonts}
\usepackage{amsmath,amsthm}

\usepackage[OT1]{fontenc}

\usepackage{bm}
\usepackage{mathrsfs}
\allowdisplaybreaks[4]

\newtheorem{thm}{Theorem}[section]
\newtheorem{lem}[thm]{Lemma}
\newtheorem{pro}[thm]{Proposition}
\numberwithin{equation}{section}

\begin{document}

\title{\textbf{Zero-density estimates for $\bm{L}$-functions attached to cusp forms}}
\author{Yoshikatsu Yashiro\\
\small Graduate School of Mathematics, Nagoya University,\\[-6pt] 
\small 464-8602 \ Chikusa-ku, Nagoya, Japan \\[-6pt] 
\small E-mail: m09050b@math.nagoya-u.ac.jp
}
\date{}
\maketitle

\renewcommand{\thefootnote}{}
\begin{abstract}
Let $S_k$ be the space of holomorphic cusp forms of weight $k$ with respect to $SL_2(\mathbb{Z})$. Let $f\in S_k$ be a normalized Hecke eigenform, $L_f(s)$ the $L$-function attached to the form $f$. 
In this paper we consider the distribution of zeros of $L_f(s)$ in the strip $\sigma\leq\text{Re }s\leq 1$ for fixed $\sigma>1/2$ with respect to the imaginary part. We study estimates of  
$$N_f(\sigma,T)=\#\{\rho\in\mathbb{C}\mid L_f(\rho)=0, \ \sigma\leq\text{Re }\rho\leq 1, \  0<\text{Im }\rho\leq T\}$$ 
for $1/2\leq\sigma\leq1$ and large $T>0$. Using the methods of Karatsuba and Voronin we shall give another proof for Ivi\'{c}'s method. 
\end{abstract}

\section{Introduction}
It is conjectured that non-trivial zeros of the Riemann zeta function $\zeta(s)$ lie on the critical line $\text{Re }s=1/2$ (in short RH). For $1/2\leq\sigma\leq 1$ and large $T>0$, let $N(\sigma,T)$ be the number of zeros of $\zeta(s)$ in the region $\sigma\leq\text{Re }s\leq 1$ and $0<\text{Im }s\leq T$. If RH is true then $N(\sigma,T)=0$ for $1/2<\sigma\leq 1$. The purpose of the zero-density theory for $\zeta(s)$ is to estimate $N(\sigma,T)$  as small as possible to support RH. For this problem Bohr and Landau \cite{B&L} began to study zero-density for $\zeta(s)$ and proved 
$N(\sigma,T)\ll T$ uniformly for $1/2\leq\sigma\leq 1$. 
Ingham \cite{IN1} improved their result to 
\begin{align}
N(\sigma,T)\ll T^{\frac{3}{2-\sigma}(1-\sigma)}(\log T)^{5} \quad (1/2\leq\sigma\leq 1). \label{NEI}
\end{align} 
Moreover Huxley \cite{HUX} improved the result (\ref{NEI}) for $\sigma\geq3/4$ to 
\begin{align}
N(\sigma,T)\ll T^{\frac{3}{3\sigma-1}(1-\sigma)}(\log T)^{44} \quad (3/4\leq\sigma\leq 1). \label{NEH}
\end{align}  
The results (\ref{NEI}) and (\ref{NEH}) give the following estimate:
\begin{align}
N(\sigma,T)\ll T^{2.4(1-\sigma)}(\log T)^{444} \quad (1/2\leq\sigma\leq 1). \label{NEK}
\end{align}
Karatsuba and Voronin \cite{KAR} gave alternative proof for (\ref{NEK}) by using the approximate functional equation for $\zeta(s)$ (see Hardy and Littlewood \cite{HL3}). 

In this paper we study the distribution of zeros of $L$-functions associated to holomorphic cusp forms. Let $S_k$ be the space of cusp forms of weight $k\in\mathbb{Z}_{\geq 12}$ with respect to the full modular group $SL_2(\mathbb{Z})$. Let $f\in S_k$ be a normalized Hecke eigenform, and $a_f(n)$ the $n$-th Fourier coefficient of $f$.  It is known that all $a_f(n)$'s are real numbers (see \cite[Chapter 6.14]{APO}) and estimated as $|a_f(n)|\leq d(n)n^{\frac{k-1}{2}}$ by Deligne \cite{DEL}, where $d(n)$ is the divisor function defined by $d(n)=\sum_{m|n}1$. The $L$-function attatched to $f$ is defined by
\begin{align}
L_f(s)=\sum_{n=1}^\infty\frac{\lambda_f(n)}{n^s}=\prod_{p\text{:prime}}\frac{1}{1-\lambda_f(p)p^{-s}+p^{-2s}} \quad (\text{Re }s>1),\label{LFD}
\end{align}
where $\lambda_f(n)=a_f(n)n^{-\frac{k-1}{2}}$. Hecke \cite{HEC} proved that $L_f(s)$ has an analytic continuation to the whole $s$-plane and the completed $L$-function
\begin{align}
\Lambda_f(s)=(2\pi)^{-s-\frac{k-1}{2}}\Gamma(s+\tfrac{k-1}{2})L_f(s)=\int_0^\infty f(iy)y^{s+\frac{k-1}{2}-1}dy \label{ALF}
\end{align}
satisfies $\Lambda_f(s)=(-1)^{k/2}\Lambda_f(1-s)$, namely,
\begin{align}
L_f(s)=\chi_f(s)L_f(1-s)  \label{FEL}
\end{align}
for all $s\in\mathbb{C}$, where $\chi_f(s)$ is defined by
\begin{align*}
\chi_f(s)=(-1)^{k/2}(2\pi)^{2s-1}\frac{\Gamma(1-s+\tfrac{k-1}{2})}{\Gamma(s+\tfrac{k-1}{2})}.
\end{align*}
By (\ref{LFD}) and (\ref{FEL}), $L_f(s)$ has no zeros in $\text{Re }s>1$ and $\text{Re }s<0$ except $s=-n+1/2\;(n\in\mathbb{Z}_{\geq k/2})$. The zeros of $L_f(s)$ in the critical strip $0\leq\text{Re }s\leq 1$ are called {\itshape non-trivial zeros}. Moreover, by (\ref{ALF}) and (\ref{FEL}), the non-trivial zeros are located symmetrically with respect to $\text{Im }s=0$ and $\text{Re }s=1/2$. 

It is opened that all of non-trivial zeros of $L_f(s)$ lie on $\text{Re }s=1/2$, which is called the Generalized Riemann Hypothesis (in short GRH). Let $N_f(\sigma,T)$ be the function defined by
\begin{align}
N_f(\sigma,T)=\#\{\rho\in\mathbb{C}\mid  L_f(\rho)=0, \ \sigma\leq\text{Re }\rho<1, \ 0<\text{Im }\rho\leq T \}. \label{NFL}
\end{align} 
As in the case of the Riemann zeta function, it is important to study the behavior of $N_f(\sigma,T)$. It is Ivi\'{c} \cite{IVI} who first proved the non-trivial estimates of $N_f(\sigma,T)$. 

The aim of this paper is to give an alternative proof of Ivi\'{c}'s estimate, namely, 
\begin{thm}\label{ZZZ1}
Let $f\in S_k$ be a normalized Hecke eigenform. For any large $T$ we have 
\begin{align}
N_f(\sigma,T)\ll\left\{\begin{array}{ll} T^{\frac{4}{3-2\sigma}(1-\sigma)+\varepsilon}, & 1/2\leq\sigma\leq 3/4,\\ T^{\frac{2}{\sigma}(1-\sigma)+\varepsilon}, & 3/4\leq\sigma\leq 1\end{array}\right.\label{XXXX}
\end{align}
uniformly $1/2\leq\sigma\leq 1$. Here and later $\varepsilon$ denotes arbitrarily positive small constant.
\end{thm}
Ivi\'{c} \cite{IVI} proof is based on the second and the sixth power moments of $L_f(s)$ due to Good \cite[Theorem]{GD2} and Jutila \cite[(4.4.2)]{JUT}.
In this paper, instead of using the higher power moments of $L_f(s)$, we follow Karatsuba and Voronin's approach and use only the approximate functional equation of $L_f(s)$ (see Lemma \ref{KOR}) and the well-known estimates of exponential sum (see Lemma \ref{KKCCA}, \ref{KKCCB}).

It is important that we can construct a set $\mathcal{E}$ of zeros of $L_f(s)$ such that the estimate of $R=N_f(\sigma,T_1)-N_f(\sigma,T_1/2)$ is reduced to that of $S(\rho)$, where $1/2\leq\sigma\leq 1$, $2\leq T_1\leq T$ and $S(\rho)$ is a function obtained by multiplying the approximate functional equation of $L_f(s)$ by $1/L_f(s)$. The existence of $\mathcal{E}$ works to obtain an estimate 
\begin{align}
R\ll (\log T_1)^{4\alpha+3}\sum_{\rho\in\mathcal{E}}|S(\rho)|^{2\alpha} \label{IREE}
\end{align}
where $\alpha$ is any fixed positive integer. The upper bound of sum of (\ref{IREE}) is obtained by the technique of Karatsuba and Voronin's calculating, in which power moment of $L_f(s)$ is not needed.

\section{Preliminary Lemmas}\label{SSPL}
To prove Theorem \ref{ZZZ1}, we need Lemmas \ref{LFMXL}--\ref{KKCCB}. First Lemmas \ref{LFMXL}--\ref{MORE} are required to show Proposition \ref{ZSET} (see Section \ref{SSPF}), which is required for the proof of Theorem \ref{ZZZ1}.

\begin{lem}\label{LFMXL}
If we write
\begin{align*}
\frac{1}{L_f(s)}=\sum_{n=1}^\infty\frac{\mu_f(n)}{n^s} \quad ({\rm Re\;} s>1),
\end{align*}
then we see that $\mu_f(n)$ is multiplicative and given by 
\begin{align*}
\mu_f(p^r)=\begin{cases}1, & r=0, \\ -\lambda_f(p), & r=1, \\ 1, & r=2, \\ 0, & r\in\mathbb{Z}_{\geq3} \end{cases} 
\end{align*}
for any prime number $p$. In addition we have $|\mu_f(n)|\leq d(n)$ for $n\in\mathbb{Z}_{\geq1}$.
\end{lem}

\begin{proof}
By expanding the right-hand side of (\ref{LFD}) and using Deligne's result, we can obtain the assertion of this lemma.
\end{proof}

\begin{lem}[\textit{The approximate functional equation of $L_f(s)$, }{\cite[KOROLLAR 2]{GD1}}]\label{KOR}
There exist $\alpha\in(0,1/2)$ and $\beta\in\mathbb{R}_{>0}$ such that
\begin{align}
\sum_{x\leq n\leq x(1+x^{-\alpha})}|a_f(n)|^2 \ll x^{k-\beta} \notag
\end{align}
and
\begin{align}
 L_f(s-\tfrac{k-1}{2})=&\sum_{n\leq y}\frac{a_f(n)}{n^s}+(2\pi)^{2s-k}\frac{\varGamma(k-s)}{\varGamma(s)}\sum_{n\leq y}\frac{a_f(n)}{n^{k-s}}+ \label{LK2}\\
&+O(|t|^{\frac{k+1}{2}-\sigma-\frac{\alpha+\beta}{2}}) \notag 
\end{align}
uniformly for $(k-1)/2\leq\sigma\leq(k+1)/2$ where $s=\sigma+it$ and $y=|t|/(2\pi)$. Moreover we have
\begin{align}
L_f(s)=\sum_{n\leq y}\frac{\lambda_f(n)}{n^s}+\chi_f(s)\sum_{n\leq y}\frac{\lambda_f(n)}{n^{1-s}}+O(|t|^{1/2-\sigma+\varepsilon}) \label{LFX}
\end{align} 
where $\chi_f(s)$ is given by (\ref{FEL}).
\end{lem}

\begin{proof}
We shall show (\ref{LFX}) by using (\ref{LK2}). Since $|a_f(n)|\ll n^{\frac{k-1}{2}+\varepsilon}$ for any $n\in\mathbb{Z}_{\geq1}$ from Deligne's result, it follows that
$$\sum_{x\leq n\leq x(1+x^{-\alpha})}|a_f(n)|^2\ll x^{k-\alpha+\varepsilon}$$
for $\alpha\in\mathbb{R}_{>0}$. If we put $\alpha=1/2-\varepsilon$ and $\beta=\alpha-\varepsilon$, then the $O$-term in (\ref{LK2}) becomes $O(|t|^{\frac{k}{2}-\sigma+\varepsilon})$.  Replacing $s$ by $s+(k-1)/2$ in (\ref{LK2}), we obtain the formula (\ref{LFX}).
\end{proof}

\begin{lem}[Moreno {\cite[Theorem 3.5]{MOR}}]\label{MORE} 
Let $N_f(T)$ be the number of zeros of $L_f(s)$ in the region $ 0\leq{\rm Re \ }\rho\leq 1$ and $0\leq{\rm Im \ }\rho\leq T$. For large $T>0$ we have
$$ N_f(T)=\frac{T}{\pi}\log\frac{T}{2\pi e}+O(\log T). $$
Furthermore, we have $ N_f(\sigma,T+1)-N_f(\sigma,T)\ll \log T$ uniformly for $1/2\leq\sigma\leq 1$.
\end{lem}

When we estimate $\sum_{r=0}^k|\sum_{N<n\leq N_1}a(n)n^{it_r}|^2$ where $t_0<t_1<\cdots<t_k$ and $a(n)$ is an arithmetic function, we use Lemmas \ref{LEM1} and \ref{LEM2}.
\begin{lem}[{\cite[Lemma IV.1.1]{KAR}}]\label{LEM1}
Let $t_0<t_1<\cdots<t_{k-1}<t_k$, $S(t)$ be a complex-valued $C^1$-class function on $[t_0,t_k]$. Set $d=\min_{0\le r\le k-1}|t_{r+1}-t_r|$. Then we have
\[ \sum_{r=0}^k|S(t_r)|^2\le\frac{1}{d}\int_{t_0}^{t_k}|S(t_k)|^2dt+2\sqrt{\int_{t_0}^{t_k}|S(t)|^2dt}\sqrt{\int_{t_0}^{t_k}|S'(t)|^2dt}. \]
\end{lem}

\begin{lem}[{\cite[Lemma IV.1.2]{KAR}}]\label{LEM2}
Let $a(n)$ be an arithmetic function. The parameters $X, X_1, N$ and $N_1$ satisfy
$0<X<X_1\leq 2X$ and $3\leq N<N_1\leq 2N$. Then we have
\[ \int_X^{X_1}\left|\sum_{N<n\leq N_1}a(n)n^{it}\right|^2dt \ll (X+N\log N)\sum_{N<n\le N_1}|a(n)|^2.\]
\end{lem}

When we calculate $\sum_{a<x\leq b}\varphi(x)e^{2\pi if(x)}$ where $\varphi(x)$ and $f(x)$ are real-valued $C^\infty$-class functions on $[a,b]$, we use Lemmas \ref{KKCCA} or \ref{KKCCB}.
\begin{lem}[{\cite[Corollary 1 of Lemma V.2.1]{KAR}}]\label{KKCCA}
Let $\varphi$ and $f$ be real-valued continuous functions on $[a,b]$. Suppose the following conditions are satisfied: 
\begin{list}{}{\leftmargin=60pt\labelwidth=108pt}
\item[(C1)] The functions $\varphi^{(2)}(x)$ and $f^{(4)}(x)$ are continuous. 
\item[(C2)] The function $f^{(2)}(x)$ satisfies $0<f^{(2)}(x)\ll 1$.
\item[(C3)] There exist the parameters $H,U\mathbb{R}_{>0}$ such that $0<b-a\leq U, \ \varphi(x)\ll H, \ \varphi^{(1)}(x)\ll H/U$.
\item[(C4)] There exists the constant $C>0$ such that $|f^{(1)}(x)|<C<1$.
\end{list}
Then we have
\[ \sum_{a<x\le b}\varphi(x)e^{2\pi if(x)}=\int_a^b\varphi(x)e^{2\pi if(x)}dx+O(H) \]
where the constant in the $O$-term depends on $C$.
\end{lem}

\begin{lem}[{\cite[Theorem II.3.1]{KAR}}]\label{KKCCB}
Let $\varphi$ and $f$ be real-valued continuous functions on $[a,b]$, suppose the following conditions:
\begin{list}{}{\leftmargin=60pt\labelwidth=108pt}
\item[(C1)] The functions $\varphi^{(2)}(x)$ and $f^{(4)}(x)$ are continuous. 
\item[(C2)] The parameters $H,A,U>0$ satisfy $1\ll A\ll U, \ 0<b-a\le U$ and  
$$ \begin{array}{lll}
\varphi(x)\ll H, & \varphi^{(1)}(x)\ll H/U, & \varphi^{(2)}(x)\ll H/U^2,\\
f^{(2)}(x)\asymp 1/A, & f^{(3)}(x)\ll 1/AU, & f^{(4)}(x)\ll 1/AU^2. 
\end{array} $$
\end{list}
Suppose that the numbers $x_n$ are determined from the equation $f^{(1)}(x_n)=n$. 
Then we have
\[ \sum_{a<x\leq b}\varphi(x)e^{2\pi if(x)}=\sum_{f^{(1)}(a)\le n\le f^{(1)}(b)}c(n)Z(n)+R \]
where
\begin{align*}
c(n)=&  \begin{cases} 1, & n\in(f^{(1)}(a),f^{(1)}(b)), \\ 1/2, & n=f^{(1)}(a) \text{ or } n=f^{(1)}(b),\end{cases}\\
Z(n)=& e^{\frac{\pi}{4}i}\frac{\varphi(x_n)}{\sqrt{f^{(2)}(x_n)}}e^{2\pi i(f(x_n)-nx_n)},\\
R=& O\left(H\left(T(a)+T(b)+\log(f^{(1)}(b)-f^{(1)}(a)+2)\right)\right),\\
T(x)=& \begin{cases} 0, & f^{(1)}(x)\in\mathbb{Z}, \\ {\rm min}(1/\|f^{(1)}(x)\|,\sqrt{A}), & f^{(1)}(x)\not\in\mathbb{Z}, \end{cases}\\
\|f^{(1)}(x)\|=& {\rm min}(\{f^{(1)}(x)\},1-\{f^{(1)}(x)\}),
\end{align*}
and $\{X\}$ is the fractional part of $X$.
\end{lem}

\section{Proof of the Theorem \ref{ZZZ1}}\label{SSPF}
In this section we shall prove Theorem \ref{ZZZ1}. By a standard argument, it is enough to get upper bound
$ R=N_f(\sigma,T_1)-N_f(\sigma,T_1/2)$ for $2\leq T_1\leq T$ and $\sigma>1/2$. Let $X=T_1^{\frac{1}{\delta}}$ where $\delta$ is a positive integer which is chosen later. Let $M_X(s)$ be a function defined by
$$M_X(s)=\sum_{m\leq X}\frac{\mu_f(m)}{m^{s}}.$$
We multiply both sides of (\ref{LFX}) by $M_X(s)$ and obtain
\begin{align}
L_f(s)M_X(s)=&1+\sum_{X<l\leq Xy}\frac{c_f(l)}{l^s}+\chi_f(s)\sum_{m\leq X}\frac{\mu_f(m)}{m^s}\sum_{n\leq y}\frac{\lambda_f(n)}{n^{1-s}}+\label{LSE2}\\
&+O(|t|^{\frac{1}{2}-\sigma}|M_X(s)|) \notag 
\end{align}
where 
\begin{align}
c_f(l)=\sum_{l=mn, \ m\leq X, \ n\leq y}\mu_f(m)\lambda_f(n)=\left\{\begin{array}{ll} 1, & l=1,\\ 0, & 2\leq l\leq X.\end{array}\right. \label{CLDD}
\end{align}
(Note that $|c_f(l)|\leq d_4(l)\ll l^\varepsilon.$) 
By a trivial estimate $M_X(s)\ll X^{1-\sigma+\varepsilon}$, we have 
\begin{align*}
|t|^{\frac{1}{2}-\sigma}|M_X(s)|\ll 
T_1^{\left(1+\frac{1}{\delta}\right)(1-\sigma)-\frac{1}{2}+\varepsilon}.
\end{align*}
If we choose $\delta$ sufficiently large, then we see that the exponent of $|t|$ in the error term of (\ref{LSE2}) becomes negative, that is,
there exists a positive constant $c$ such that $|t|^{\frac{1}{2}-\sigma}|M_X(s)|\ll T_1^{-c}$. Therefore taking $s=\rho\in\mathcal{U}$ in (\ref{LSE2}) where $\mathcal{U}$ is a set of zeros of $L_f(s)$ in $\sigma\leq\text{Im }s\leq 1$ and $T_1/2<\text{Im }s\leq T_1$, we get 
\begin{align}
1/2\leq&1+O(T_1^{-c})&\label{SEE}\\
\leq&\left|\sum_{X<l\leq Xy}\frac{c(l)}{l^\rho}\right|+|\chi_f(\rho)|\left|\left.\left.\sum_{m\leq X}\frac{\mu_f(m)}{m^\rho}\right|\right|\sum_{n\leq y}\frac{\lambda_f(n)}{n^{1-\rho}}\right|\notag
\end{align}
for sufficiently large $T_1$.

Dividing the intervals of summations over $l,m,n$ into subintervals of the form $(Z,2Z]$ (the last subintervals are of the form $(Z,Z_0]$ where $Z<Z_0\leq 2Z$ and $Z_0=Xy, X, y$ respectively), we can write (\ref{SEE}) as 
\begin{align}
\frac{1}{2}\leq\sum_{\nu=1}^D|S_\nu(\rho)| \label{DDDX}
\end{align}
where
\begin{subequations}
\begin{align}
S_\nu(\rho)=&\sum_{L_\nu<l\leq L_\nu'}\frac{c(l)}{l^\rho},\label{S1E}\\
S_\nu(\rho)=&\chi_f(\rho)\sum_{M_\nu<m\leq M_\nu'}\frac{\mu_f(m)}{m^\rho}\sum_{N_\nu<n\leq N_\nu'}\frac{\lambda_f(n)}{n^{1-\rho}}.\label{S2E}
\end{align}
\end{subequations}
(Note that the number of summands of (\ref{DDDX}) is $\ll(\log T_1)^2$, namely $D\ll(\log T_1)^2$.)

Now following Karatsuba and Voronin, we shall show the existence of a set $\mathcal{E}$ of zeros of $L_f(s)$ with playing an important role later.
\begin{pro}\label{ZSET}
There exists a set $\mathcal{E}$ of zeros of $L_f(s)$ such that
\begin{subequations}
\begin{align}
& |S(\rho)|\geq\frac{1}{2D} \quad (\rho\in\mathcal{E}),\label{1SE} \\
& \#\mathcal{E}\gg\frac{R}{D\log T_1}, \label{2SE} \\
& |{\rm Im}\;\rho-{\rm Im}\;\rho'|\geq 1 \quad (\rho\ne\rho'\in\mathcal{E}), \label{3SE}
\end{align}
\end{subequations}
where $S(\rho)=S_{\nu_0}(\rho)$ with some number $\nu_0\in\{1,\dots, D\}$.
\end{pro}
\begin{proof}
From (\ref{DDDX}) it is clear that
$ \mathcal{U}=\bigcup_{1\leq \nu\leq D}\mathcal{A}_\nu$  where  $\mathcal{A}_\nu=\{\rho\in \mathcal{U} \mid |S_\nu(\rho)|\geq 1/(2D)\}$.
Then we see that there exists $\nu_0$ such that $|S(\rho)|\geq 1/(2D)$ for $\rho\in\mathcal{A}$ and $\#\mathcal{A}\geq R/D$ where $\mathcal{A}=\mathcal{A}_{\nu_0}$ and $S(\rho)=S_{\nu_0}(\rho)$.

Let $\rho_{m,n}$ be $\rho\in\mathcal{A}$ such that $\text{Im }\rho$ is the $n$-th minimum number in $(T_1/2+m,T_1/2+m+1]$. By using Lemma \ref{MORE} we can write 
\begin{align*}
\mathcal{A}=\bigcup_{j=0,1}\bigcup_{1\leq n\leq C\log T_1}\mathcal{E}_{n,j}, \quad \mathcal{E}_{n,j}=\{\rho_{j,n},\rho_{2+j,n},\rho_{4+j,n},\dots,\rho_{2[T_1/4]+j,n}\}
\end{align*}
where $C$ is a positive constant. Then there exist $n_0\in\{1,2,\dots,[C\log T_1]\}$ and $j_0\in\{0,1\}$ such that $\#\mathcal{A}\leq C(\log T_1)\sum_{j=0,1}\#\mathcal{E}_{n_0,j}\leq 2C(\log T_1)\#\mathcal{E}$ where $\mathcal{E}=\mathcal{E}_{n_0,j_0}$. Since $\mathcal{E}\subset\mathcal{A}$ and $\#\mathcal{A}\geq R/D$, it follows that (\ref{1SE}) and (\ref{2SE}) are shown. And (\ref{3SE}) is shown because $|\text{Im }\rho_{2l+j_0,n_0}-\text{Im }\rho_{2l'+j_0,n_0}|\geq 1$ for $l\ne l'\in\{0,1,\dots,[T_1/4]\}$.
\end{proof}
From (\ref{1SE}) and (\ref{2SE}) we can reduce the estimate of $R$ to that of $S(\rho)$, that is, by taking $2\alpha$-th power of both sides of (\ref{1SE}) we get
\begin{align}
R\ll D(\log T_1)\#\mathcal{E}
=D(\log T_1)\sum_{\rho\in\mathcal{E}}1^{2\alpha} 
\ll&(\log T_1)^{4\alpha+3}\sum_{\rho\in\mathcal{E}}|S^\alpha(\rho)|^2 \label{EEEE}
\end{align}
where $\alpha$ is any fixed positive integer. 

First we consider the case that $S(\rho)$ is of the form (\ref{S1E}). We shall give a preliminary upper bound of $R$ as
\begin{pro}\label{NFER1}
Let $S(\rho)$ in (\ref{EEEE}) has the form
\begin{align}
S(\rho)=\sum_{L<l\leq L'}\frac{c_f(l)}{l^\rho} \label{SSRR}
\end{align}
with $L<L'\leq 2L$. For any positive integer $\alpha$ we have
\begin{align}
R\ll L^{\alpha(1-2\sigma)}(T_1+L^\alpha)T_1^{\varepsilon}. \label{NFERR1}
\end{align}
\end{pro}
\begin{proof}
From (\ref{SSRR}) the $\alpha$-th power of $S(\rho)$ has the form
\begin{align*}
S^\alpha(\rho)=\sum_{L^\alpha<l\leq L'^\alpha}\frac{A_\alpha(l)}{l^\rho}=\sum_{L^\alpha<l\leq {L'}^\alpha}\frac{A_\alpha(l)}{l^{\beta+i\gamma}}, \end{align*}
where
\begin{align*}
A_\alpha(l)=\sum_{\begin{subarray}{c} l=l_1\cdots l_\alpha, L<l_1,\cdots,l_\alpha\leq L' \end{subarray}}c_f(l_1)\cdots c_f(l_\alpha).
\end{align*}  
(Note that $|A_\alpha(l)|\leq d_{4\alpha}(l)\ll l^\varepsilon$.) If we put $C(t)=\sum_{L^\alpha<l\leq t}A_\alpha(l)l^{-i\gamma}$, then by partial summation formula and Cauchy's inequality we have
\begin{align}
|S^\alpha(\rho)|^2
=&\left|\frac{C({L'}^\alpha)}{({L'}^\alpha)^\beta}-\beta\int_{L^\alpha}^{{L'}^\alpha}\frac{C(t)}{t^{\beta+1}}dt\right|^2\notag\\
\ll&\frac{|C({L'}^\alpha)|^2}{L^{2\alpha\sigma}}+\frac{1}{L^{2\alpha(\sigma+1)}}\left(\int_{L^\alpha}^{{L'}^\alpha}|C(t)|dt\right)^2\notag\\
\ll&\frac{1}{L^{2\alpha\sigma}}\left(|C({L'}^\alpha)|^2+\frac{1}{L^\alpha}\int_{L^\alpha}^{{L'}^\alpha}|C(t)|^2dt\right)\label{RRRA}
\end{align}
Hence from (\ref{EEEE}) and (\ref{RRRA}) we obtain
\begin{align}
R\ll\frac{(\log T_1)^{4\alpha+3}}{L^{2\alpha\sigma}}\sum_{\rho\in\mathcal{E}}|C(L_0)|^2 \label{RTE1}.
\end{align}
where $L_0$ is chosen such that $\sum_{\rho\in\mathcal{E}}|C(L_0)|^2$ is the maximal value. To estimate this maximal value, we divide again the interval $(L^\alpha,L_0]$ into the  subintervals of the form $(Z,Z']$ where $Z<Z'\leq 2Z$ and let $L_1\in(L^\alpha,L_0]$ be chosen so that $\sum_{\rho\in\mathcal{E}}|\sum_{L_1<l\leq L_1'}A_\alpha(l)l^{i\gamma}|$ is maximal. (Note that the number of divided interval is at most $\alpha$ because $L_0\leq 2^\alpha L^\alpha$.) Then we have
\begin{align}
\sum_{\rho\in\mathcal{E}}|C(L_0)|^2
\ll \sum_{\rho\in\mathcal{E}}\left|\sum_{L_1<l\leq L_1'}A_\alpha(l)l^{i\gamma}\right|^2.
\label{RTE2}
\end{align}
Now we apply Lemma \ref{LEM1} to the right hand side of the above formula:
\begin{align}
\sum_{\rho\in\mathcal{E}}\left|\sum_{L_1<l\leq L_1'}A_\alpha(l)l^{i\gamma}\right|^2 \ll I_1+\sqrt{I_1I_2},\label{RTE3}
\end{align}
where 
\begin{align*}
I_1=\int_{\frac{T_1}{2}}^{T_1}\left|\sum_{K_{\nu_0}<l\leq K_{\nu_0}'}A_\alpha(l)l^{i\gamma}\right|^2d\gamma, \ 
I_2=\int_{\frac{T_1}{2}}^{T_1}\left|\sum_{K_{\nu_0}<l\leq K_{\nu_0}'}A_\alpha(l)l^{i\gamma}\log l\right|^2d\gamma. 
\end{align*}
By using Lemma \ref{LEM2} and noting $|A_\alpha(l)|\ll l^\varepsilon$, we see that
\begin{align}
I_1\ll L(T_1+L)T_1^{\varepsilon}, \quad I_2\ll L(T_1+L)T_1^{\varepsilon}. \label{RTE4}
\end{align} 
Combining (\ref{RTE1})--(\ref{RTE4}) we obtain the assertion of Proposition \ref{NFER1}.
\end{proof}

We divide the interval $(X,Xy]$ into subintervals of the form:
\begin{align}
\mathcal{F}_r=\begin{cases} (T_1,T_1^{1+\frac{1}{\delta}}], & r=1, \\ (T_1^{\frac{1}{r}},T_1^{\frac{1}{r-1}}], & r\in\{2,\dots,\delta\}. \end{cases} \label{FFRR}
\end{align}
We see that there exists $r\in\{1,2,\dots,\delta\}$ such that $L\in\mathcal{F}_r$. If we use (\ref{NFERR1}) under some conditions of $L$, then we can obtain the following upper bounds of $R$:
\begin{pro}\label{PROEA}
We have
\begin{align*}
R\ll \begin{cases} T_1^{2(1-\sigma)+\varepsilon}, & L\in\mathcal{F}_1, \\ T_1^{\frac{4}{3-2\sigma}(1-\sigma)+\varepsilon}, & L\in\bigcup_{2\leq r\leq\delta}\mathcal{F}_r. \end{cases}
\end{align*}
\end{pro}
\begin{proof}
First in the case of $L\in\mathcal{F}_1$, taking $\alpha=1$ in (\ref{NFERR1}) and choosing $\varepsilon\geq 2/\delta$, we  obtain
\begin{align}
R\ll L^{2(1-\sigma)}T_1^\varepsilon\ll T_1^{2\left(1+\frac{1}{\delta}\right)(1-\sigma)+\varepsilon}\ll T_1^{2(1-\sigma)+\varepsilon}. \label{EEX1}
\end{align} 

Next we consider the upper bound of $R$ in the case of $L\in\bigcup_{2\leq r\leq\delta}\mathcal{F}_r$. 
Let $2\leq A\leq 4$. When $r\in\{2,\dots,[A/(A-2)]\}$, that is, $A\leq 2r/(r-1)$,
we can divide $\mathcal{F}_r$ to $(T_1^{\frac{1}{r}},T_1^{\frac{A}{2r}}]$ and $(T_1^{\frac{A}{2r}},T_1^{\frac{1}{r-1}}]$. Hence $L\in(T_1^{\frac{1}{r}},T_1^{\frac{A}{2r}}]$, taking $\alpha=r$ in (\ref{NFERR1}) we have
\begin{align}
R\ll L^{2r(1-\sigma)}T_1^\varepsilon
\ll T_1^{A(1-\sigma)+\varepsilon}. \label{EEX2}
\end{align}
When $r\in\{[A/(A-2)]+1,\dots,\delta\}$ and $L\in\mathcal{F}_r$, we see that $A>2r/(r-1)$ and 
\begin{align}
R\ll L^{2r(1-\sigma)}T_1^\varepsilon \ll T_1^{\frac{2r}{r-1}(1-\sigma)+\varepsilon}\ll T_1^{A(1-\sigma)+\varepsilon}. \label{EEX3}
\end{align}

On the other hand, in the case of $L\in\bigcup_{2\leq r<A/(A-2)}(T_1^{\frac{A}{2r}},T_1^{\frac{1}{r-1}}]$, taking $\alpha=r-1$ in (\ref{NFERR1}) we obtain  
\begin{align*}
R\ll L^{(r-1)(1-2\sigma)}T_1^{1+\varepsilon}\ll T_1^{1+\frac{r-1}{2r}A(1-2\sigma)+\varepsilon}.
\end{align*}
Here we suppose $1+A(1-2\sigma)(r-1)/2r\leq A(1-\sigma)$, that is, $A\geq ((2\sigma-1)(r-1)/2r+(1-\sigma))^{-1}$. 
If we put $A=\max_{r\geq 2}((2\sigma-1)(r-1)/2r+(1-\sigma))^{-1}=4/(3-2\sigma)$, then we obtain
\begin{align}
R\ll T_1^{A(1-\sigma)+\varepsilon}\ll T_1^{\frac{4}{3-2\sigma}(1-\sigma)+\varepsilon}. \label{EEX4}
\end{align}
From the results (\ref{EEX1})--(\ref{EEX4}), the proof of Proposition \ref{NFER1} is completed.
\end{proof}

We consider an improvement of estimate of $R$ when $L\in\bigcup_{2\leq r\leq\delta}\mathcal{F}_r$. We replace $2\alpha$ by $r-1$ in (\ref{S1E}), then
\begin{align}
R\ll (\log T_1)^{2r+1}\sum_{\rho\in\mathcal{E}}|S^{r-1}(\rho)|. \label{R2EE}
\end{align}
By writing $|S^{r-1}(\rho)|=S^{r-1}(\rho)e^{-i\theta(\rho)}$  and using Cauchy's inequality, (\ref{R2EE}) becomes
\begin{align}
R\ll &(\log T_1)^{2r+1}\sum_{\rho\in\mathcal{E}}e^{-i\theta(\rho)}\sum_{L^{r-1}<l\leq L'^{r-1}}\frac{A_{r-1}(l)}{l^{i\gamma}}\ll\sqrt{W}\sqrt{W'}T_1^\varepsilon\label{RRNN1}
\end{align}
where $\theta(\rho)=\arg S(\rho)$, and
\begin{align*}
W=\sum_{L^{r-1}<l\leq L'^{r-1}}\left|\sum_{\rho\in\mathcal{E}}\frac{e^{-i\theta(\rho)}}{l^\rho}\right|^2, \quad W'=\sum_{L^{r-1}<l\leq L'^{r-1}}|A_{r-1}(l)|^2.
\end{align*}
By noting $|A_{r-1}(l)|\ll l^\varepsilon$, it is obvious that 
\begin{align}
W'\ll L^{r-1}T_1^\varepsilon. \label{ZZWP}
\end{align}
From (\ref{3SE}) we see that $\rho=\rho'$ when $\gamma=\gamma'$, and obtain
\begin{align}
W=&\sum_{L^{r-1}< l \leq{L'}^{r-1}}\sum_{\rho,\rho'\in\mathcal{E}}\frac{e^{-i(\theta(\rho)-\theta(\rho'))}}{l^{\beta+\beta'+i(\gamma-\gamma')}}
\nonumber\\
=&\sum_{\rho\in\mathcal{E}}\sum_{L^\alpha< l \leq{L'}^\alpha }\frac{1}{l^{2\beta}}+\sum_{\gamma\ne\gamma'}e^{-i(\theta(\rho)-\theta(\rho'))}\sum_{L^\alpha< l \leq{L'}^\alpha } \frac{l^{i(\gamma'-\gamma)}}{l^{\beta+\beta'}}.
\label{ZZZW}
\end{align}
We shall calculate two terms of (\ref{ZZZW}). Since $\#\mathcal{E}\leq R$, it follows that
\begin{align}
\sum_{\rho\in\mathcal{E}}\sum_{L^{r-1}< l \leq{L'}^{r-1} }\frac{1}{l^{2\beta}} \ll RL^{(r-1)(1-2\sigma)}. \label{ZZS1}
\end{align} 
Using partial summation formula and putting  $C_{\gamma,\gamma'}(t)=\sum_{L^{r-1}<l\leq t}l^{i(\gamma'-\gamma)} \label{CCC0}$ , we see that 
\begin{align}
\sum_{\gamma\ne\gamma'}e^{-i(\theta(\rho)-\theta(\rho'))}\sum_{L^{r-1}< l \leq{L'}^{r-1} } \frac{l^{i(\gamma'-\gamma)}}{l^{\beta+\beta'}}
\ll L^{-2(r-1)\sigma}\sum_{\gamma\ne\gamma'}|C_{\gamma,\gamma'}(L_0)| \label{ZZS2}
\end{align}
where $L_0$ is chosen such that $\sum_{\gamma\ne\gamma'}|\sum_{L^{r-1}<l\leq L_0}l^{i(\gamma'-\gamma)}|$ is the maximal value. In order to estimate the maximal sum, we shall divide the above sum as 
\begin{align*}
\sum_{\gamma\ne\gamma'}=\sum_{1<|\gamma-\gamma'|\leq 2}+\sum_{2<|\gamma-\gamma'|\leq 4}+\cdots+\sum_{2^{D_1}<|\gamma-\gamma'|\leq T_1/2}. 
\end{align*}
(Note that the number of divided sums is $\ll\log T_1$, that is, $D_1\ll\log T_1$.) Then we see that
\begin{align}
\sum_{\gamma\ne\gamma'}|C_{\gamma,\gamma'}(L_0)|\ll (\log T_1)\sum_{V<|\gamma-\gamma'|\leq V'}|C_{\gamma,\gamma'}(L_0)| \label{ZZCC12}
\end{align}
where $V\in[1,T_1/2]$ and $V<V'\leq 2V$. By fixing $\gamma'=\gamma_1=\min_{\rho\in\mathcal{E}}\gamma$, we get
\begin{align}
\sum_{V<|\gamma-\gamma'|\leq V'}|C_{\gamma,\gamma'}(L_0)|\ll R\sum_{V<\gamma-\gamma_1\leq V'}|C_{\gamma,\gamma_1}(L_0)| \label{ZZC2}.
\end{align}
Combining the above (\ref{ZZZW})--(\ref{ZZC2}), we have
\begin{align}
W\ll RL^{(r-1)(1-2\sigma)}+RL^{-2(r-1)\sigma}T_1^\varepsilon \sum_{V<\gamma-\gamma_1\leq V'}|C_{\gamma,\gamma_1}(L_0)|.  \label{WWW}
\end{align}
Taking squares in both sides of (\ref{RRNN1}), from (\ref{ZZWP}) and (\ref{WWW}) we obtain
\begin{align}
R\ll L^{2(r-1)(1-\sigma)}T_1^\varepsilon+L^{(r-1)(1-2\sigma)}T_1^\varepsilon \sum_{V<\gamma-\gamma_1\leq V'}|C_{\gamma,\gamma_1}(L_0)|.\label{WWF}
\end{align}

By estimating the sum of the second term of (\ref{WWF}) in the case of $2V\leq\pi L^{r-1}$ or $2V>\pi L^{r-1}$, we can give upper bounds of $R$ as
\begin{pro}\label{PROEB}
We have
\begin{align*}
R\ll\begin{cases} T_1^{2(1-\sigma)+\varepsilon}, & L\in\bigcup_{2\leq r\leq \delta}\mathcal{F}_r \text{ and } 2V\leq\pi L^{r-1}, \\
T_1^{\frac{2}{\sigma}(1-\sigma)+\varepsilon}, & L\in\bigcup_{2\leq r\leq \delta}\mathcal{F}_r, \ 2V>\pi L^{r-1}\text{ and }2/3\leq\sigma\leq1. \end{cases}
\end{align*}
\end{pro}
\begin{proof}
First we consider the upper bound of sum of (\ref{WWF}) in the case of $2V\leq\pi L^{r-1}$, namely, $1\ll L^{r-1}/V$. By using Lemma \ref{KKCCA} we see that
\begin{align*}
C_{\gamma,\gamma_1}(L_0)=\int_{L^{r-1}}^{L_0}u^{-i(\gamma-\gamma_1)}du+O(1)
\ll \frac{L^{r-1}}{\gamma-\gamma_1}+1\ll \frac{L^{r-1}}{V}.
\end{align*}
This formula and (\ref{WWF}) imply that 
\begin{align}
R\ll L^{2(r-1)(1-\sigma)}T_1^\varepsilon\left(1+\sum_{V<\gamma-\gamma_1\leq V'}\frac{1}{V}\right)\ll T_1^{2(1-\sigma)+\varepsilon} 
\label{RREEA}
\end{align}
for $L\in\mathcal{F}_r\;(2\leq r\leq\delta)$.

Next in the case of $2V>\pi L^{r-1}$ we apply the estimate of $C_{\gamma,\gamma_1}(L_0)$ to Lemma \ref{KKCCB}, then
\begin{align}
C_{\gamma,\gamma_1}(L_0)=e^{i\left(\frac{\pi}{4}-(\gamma-\gamma_1)\log\frac{\gamma-\gamma_1}{2\pi e}\right)}\sqrt{\frac{\gamma-\gamma_1}{2\pi}}\sum_{N_1\leq n\leq N_2}\frac{n^{i(\gamma-\gamma_1)}}{n}
+O\left(\frac{L^{r-1}}{\sqrt{V}}\right) \label{C0L0S}
\end{align}
where $N_1=(\gamma-\gamma_1)/((2\pi L_0)$ and $N_2=(\gamma-\gamma_1)/(2\pi L^{r-1})$. 
We shall calculate the sum on the right hand side of (\ref{C0L0S}). Since 
\begin{align*}
\sum_{N_1\leq n\leq N_2}\frac{n^{i(\gamma-\gamma_1)}}{n}=
&\frac{1}{B} \sum_{\begin{subarray}{c} b\ne 0, \\ -\frac{B}{2}+1\leq b \leq \frac{B}{2}  \end{subarray} } \sum_{N_3\leq n\leq N_4}\frac{n^{i(\gamma-\gamma_1)}e^{\frac{2\pi ibn}{B}}}{n} \sum_{N_1\leq m \leq N_2}e^{-\frac{2\pi ibm}{B}}+\\
&+\frac{1}{B}\sum_{N_3\leq n\leq N_4}\frac{n^{i(\gamma-\gamma_1)}}{n} \sum_{N_1\leq m \leq N_2}  1
\end{align*}
 and
\begin{align*}
\left|\sum_{N_1\leq m \leq N_2}e^{-\frac{2\pi ibm}{B}}\right|
&\leq\frac{2}{|1-e^{-\frac{2\pi ib}{B}}|}
\leq \frac{B}{|b|} \quad (b\ne 0),
\end{align*}
where $N_3=V/(2\pi L_0), N_4=V/(\pi L^{r-1})$, $B=2[V/L^{r-1}]$   
(note that $N_2-N_1<N_4-N_3<B$,  
$N_3\ll VL^{-(r-1)}$ and $N_4\ll VL^{-(r-1)}$), 
it follows that 
\begin{align}
&\sum_{N_1\leq n\leq N_2}\frac{n^{i(\gamma-\gamma_1)}}{n}\label{EBBN}\\
&\ll\sum_{\begin{subarray}{c} b\ne 0, \\ -\frac{B}{2}+1\leq b \leq \frac{B}{2} \end{subarray}}\frac{1}{|b|}\left| \sum_{N_3\leq n\leq N_4}\frac{n^{i(\gamma-\gamma_1)}e^{\frac{2\pi ibn}{B}}}{n}\right|+\left|\sum_{N_3\leq n\leq N_4}\frac{n^{i(\gamma-\gamma_1)}}{n}\right|\nonumber\\
&\ll\sum_{ -\frac{B}{2}+1\leq b \leq \frac{B}{2}}\frac{1}{1+|b|}\left| \sum_{N_3\leq n\leq N_4}\frac{n^{i(\gamma-\gamma_1)}e^{\frac{2\pi ibn}{B}}}{n}\right|. \notag
\end{align}
Combining the results (\ref{C0L0S}) and (\ref{EBBN}) we have 
\begin{align}
&\sum_{V<\gamma-\gamma_1\leq V'}|C_{\gamma,\gamma_1}(L_0)| \label{SSEE}\\
&\ll \sqrt{V}T_1^\varepsilon\sum_{V<\gamma-\gamma_1\leq V'}X\left(\left| \sum_{N_3\leq n\leq N_4}\frac{n^{i(\gamma-\gamma_1)}e^{\frac{2\pi ib_0n}{B}}}{n}\right|+\frac{L^{r-1}}{V}\right) \notag
\end{align}
where $b_0$ is chosen such that $\sum_{V<\gamma-\gamma_1\leq 2V}|\sum_{N_3\leq n\leq N_4}n^{i(\gamma-\gamma_1)-1}e^{\frac{2\pi ib_0n}{B}}|$ is the maximal value.
Here we apply (\ref{1SE}) to $X$ in (\ref{SSEE}), that is,
\begin{align}
X=1\ll T_1^\varepsilon\left|\sum_{L^{r-1}<l\leq L'^{r-1}}\frac{A_{r-1}(l)l^{i\gamma}}{l^{\beta}}\right|. \label{SSE1}
\end{align}
From (\ref{WWF}), (\ref{SSEE}) and (\ref{SSE1}) we obtain
\begin{align}
R\ll& L^{2(r-1)(1-\sigma)}T_1^\varepsilon+  \label{SSER}\\
&+L^{(r-1)(1-2\sigma)}\sqrt{V} T_1^\varepsilon \sum_{V<\gamma-\gamma_1\leq V'}\left|\sum_{L^{r-1}<l\leq L'^{r-1}}\frac{A_{r-1}(l)l^{i\gamma}}{l^{\beta}}\right|\times \notag\\
&  \times \left(\left| \sum_{N_3\leq n\leq N_4}\frac{n^{i(\gamma-\gamma_1)}e^{\frac{2\pi ib_0n}{B}}}{n}\right|+\frac{L^{r-1}}{V}\right).   \notag
\end{align}
Put $C_1(t)=\sum_{L^{r-1}<l\leq t}A_{r-1}(l)l^{i\gamma}$ and $C_2(u)=\sum_{N_3\leq n\leq u}n^{i(\gamma-\gamma_1)}e^{\frac{2\pi ib_0n}{B}}$.
Using partial summation formula, and Cauchy's inequality, we have
\begin{align}
&\sum_{V<\gamma-\gamma_1\leq V'}
\left|\sum_{L^{r-1}<l\leq L'^{r-1}}\frac{A_{r-1}(l)l^{i\gamma}}{l^{\beta}}\right|\left(\left| \sum_{N_3\leq n\leq N_4}\frac{n^{i(\gamma-\gamma_1)}e^{\frac{2\pi ib_0n}{B}}}{n}\right|+\frac{L^{r-1}}{V}\right)\label{RRRWWW}\\
&\ll L^{(r-1)(1-3\sigma)}V^{-\frac{1}{2}}T_1^\varepsilon
(W_1+W_2) \notag
\end{align}
where
\begin{align*}
W_1=\sum_{V<\gamma-\gamma_1\leq V'}|C_1(L_1)|, \quad
W_2=\sum_{V<\gamma-\gamma_1\leq V'}|C_1(L_2)||C_2(N_5)|,
\end{align*}
and $L_1$, $L_2, N_5$ are chosen such that $W_1$, $W_2$ are the maximal values.

We shall calculate $W_1$ and $W_2$. Writing 
$$|C_1(L_1)|=C_1(L_1)e^{-i\theta_1(\gamma)}, \quad |C_1(L_2)||C_2(N_5)|=C_1(L_2)C_2(N_5)e^{-\theta_2(\gamma)}$$ and using Cauchy's inequality we obtain
\begin{align}
W_1\ll\sqrt{S_{1,1}}\sqrt{S_{1,2}}, \quad W_2\ll\sqrt{S_{2,1}}\sqrt{S_{2,2}} \label{WWEE}
\end{align}
where  
\begin{align*}
&S_{1,1}=\sum_{L^{r-1}<l\leq L_1}|A_{r-1}(l)|^2, \quad S_{1,2}=\sum_{L^{r-1}<l\leq L_1}\left|\sum_{V<\gamma-\gamma_1\leq V'}e^{-i\theta_1(\gamma)}l^{i(\gamma-\gamma_1)}\right|^2,\\
&S_{2,1}=\sum_{M_1<m\leq M_2}B(m)^2, 
\hspace{11mm} S_{2,2}=\sum_{M_1<m\leq M_2}\left|\sum_{V<\gamma-\gamma_1\leq V'}e^{-i\theta_2(\gamma)}m^{i(\gamma-\gamma_1)}\right|^2,\\
&B(m)=\sum_{\begin{subarray}{c}l|m, \ L^{r-1}<l\leq L_1, \ N_3<m/l\leq N_5\end{subarray}}|A_{r-1}(l)|, 
\end{align*}
$\theta_1(\gamma)=\arg C_1(L_1)$, $\theta_2(\gamma)=\arg C_1(L_2)C_2(N_5)$, $M_1=V/2^{r+2}$, $M_2=2^{r-2}V$. (Note that $(L^{r-1}N_3,L_2N_5]\subset(M_1,M_2]$.)
It is trivial that
\begin{align}
S_{1,1}\ll VT_1^\varepsilon, \quad S_{2,1}\ll VT_1^\varepsilon \label{SS12}.
\end{align}
By using Lemma \ref{LEM1} we get
\begin{align}
S_{1,2}\ll I_{1,2}+\sqrt{I_{1,2}J_{1,2}}, \quad S_{2,2}\ll I_{2,2}+\sqrt{I_{2,2}J_{2,2}} \label{SSSS}
\end{align}
where
\begin{align*}
 I_{1,2}=&\int_{L^{r-1}}^{L_1}\left| \sum_{V<\gamma-\gamma_1\leq V'}e^{-i\theta_1(\gamma)}u^{i(\gamma-\gamma_1)}\right|^2du,\\
 J_{1,2}=&\int_{L^{r-1}}^{L_1}\left| \sum_{V<\gamma-\gamma_1\leq V'}e^{-i\theta_1(\gamma)}u^{i(\gamma-\gamma_1)-1}i(\gamma-\gamma_1)\right|^2du,\\
 I_{2,2}=&\int_{M_1}^{M_2}\left|\sum_{V<\gamma-\gamma_1\leq V'}e^{-i\theta_2(\gamma)}u^{i(\gamma-\gamma_1)}\right|^2du,\\
 J_{2,2}=&\int_{M_1}^{M_2}\left|\sum_{V<\gamma-\gamma_1\leq V'}e^{-i\theta_2(\gamma)}u^{i(\gamma-\gamma_1)-1}i(\gamma-\gamma_1)\right|^2du.
 \end{align*}
Here we shall calculate the above sums as 
$$\int_I\left|\sum_{\gamma}f_\gamma(u)\right|^2du=\sum_{\gamma=\gamma'}\int_I f_\gamma(u)\overline{f_{\gamma'}(u)}du+\sum_{\gamma\ne\gamma'}\int_I f_\gamma(u)\overline{f_{\gamma'}(u)}du,$$
and use the trivial estimates $|\gamma-\gamma_1|\ll V$, $|L_1-L^{r-1}|\ll L^{r-1}$, $M_1\gg V$, $|M_2-M_1|\ll V$ and $\# E\ll R$.
Then we can obtain
\begin{align}
I_{1,2}&\ll RL^{r-1}T_1^\varepsilon, & J_{1,2}&\ll RV^2L^{1-r}T_1^\varepsilon, &\label{IJIJ}\\
I_{2,2}&\ll RVT_1^\varepsilon,  & J_{2,2}&\ll RVT_1^\varepsilon. &\notag
\end{align}
Hence from (\ref{WWEE})--(\ref{IJIJ}) and the condition $2V>\pi L^{r-1}$ we have
\begin{align}
W_1\ll \sqrt{R}VT_1^\varepsilon, \quad W_2\ll \sqrt{R}VT_1^\varepsilon. \label{WE12}
\end{align}
Combining the results (\ref{SSER}), (\ref{RRRWWW}) and (\ref{WE12}) we obtain 
\begin{align}
R\ll& L^{2(r-1)(1-\sigma)}T_1^\varepsilon+L^{(r-1)(1-3\sigma)}\sqrt{RV}T_1^\varepsilon\label{RREE}\\
 \ll& L^{2(r-1)(1-\sigma)}(1+\sqrt{RV}L^{-(r-1)\sigma})T_1^\varepsilon.\notag
\end{align}

Moreover we shall calculate the right hand side of (\ref{RREE}). In the case of $1\gg\sqrt{RV} L^{-(r-1)\sigma}$ and $r\in\{2,\dots,\delta\}$, (\ref{RREE}) becomes
\begin{align}
R\ll L^{2r(1-\sigma)}T_1^\varepsilon\ll T_1^{2(1-\sigma)+\varepsilon}\label{RREE1}
\end{align}
for $L\in\mathcal{F}_r$.
On the other hand, in the case of $1\ll\sqrt{RV}L^{-(r-1)\sigma}$ and $r\in\{2,\dots,\delta\}$, (\ref{RREE}) becomes $R\ll L^{(r-1)(2-3\sigma)}\sqrt{RV}T_1^\varepsilon$, that is,
\begin{align}
R\ll VL^{2(r-1)(2-3\sigma)}T_1^\varepsilon
\ll T_1^{1+\frac{r-1}{r}A(2-3\sigma)+\varepsilon} \label{RREEX} 
\end{align}
for $\sigma\geq 2/3$ and $L\in(T_1^{\frac{A}{2r}},T_1^{\frac{1}{r-1}}]$.
Here we suppose $1+A(2-3\sigma)(r-1)/r\leq A(1-\sigma)$, that is, $A\geq(1-\sigma+(3\sigma-2)(1-1/r))^{-1}$. If we put $A=\max_{r\geq 2}(1-\sigma+(3\sigma-2)(1-1/r))^{-1}=2/\sigma$ where $2/3\leq\sigma\leq 1$, then from (\ref{RREEX}) we obtain
\begin{align}
R\ll T_1^{\frac{2}{\sigma}(1-\sigma)+\varepsilon} \quad (2/3\leq\sigma\leq 1, \textstyle L\in\bigcup_{2\leq r\leq\delta}(T_1^{\frac{A}{2r}},T_1^{\frac{1}{r-1}}]). \label{RREE2}
\end{align}
Combining the results (\ref{RREEA}), (\ref{RREE1}) and (\ref{RREE2}), the proof of Proposition \ref{PROEB} is completed.
\end{proof}

Finally we consider in the case of $S(\rho)=S_\nu(\rho)$ (see (\ref{S2E})). We shall give an upper bound of $R$:
\begin{pro}\label{PROEC}
Let $S(\rho)$ in (\ref{EEEE}) has the form
\begin{align*}
S(\rho)=\chi_f(\rho)\sum_{M<m\leq M'}\frac{\mu_f(m)}{m^\rho}\sum_{N<n\leq N'}\frac{\lambda_f(n)}{n^{1-\rho}}
\end{align*}
with $M<M'\leq 2M$ and $N<N'\leq 2N$. When $\alpha=1$, we have
\begin{align*}
R\ll T_1^{2(1-\sigma)+\varepsilon} \quad (1/2<\sigma\leq 1).
\end{align*}
\end{pro}

\begin{proof}
If we take $\alpha=1$ in (\ref{EEEE}), then we have
\begin{align}
R\ll(\log T_1)^7\sum_{\rho_\in\mathcal{E}}|\chi_f(\rho)|^2\left|\sum_{M<m\leq M'}\frac{\mu_f(m)}{m^\rho}\sum_{N<n\leq N'}\frac{\lambda_f(n)}{n^{1-\rho}}\right|^2.\label{RR2A}
\end{align}
Here using Stirling's formula with the condition $T_1/2<\gamma\leq T_1$ we have
\begin{align}
|\chi_f(\rho)|=(2\pi)^{2\beta-1}\frac{\gamma^{1-\beta+\frac{k-1}{2}}e^{-\frac{\pi}{2}}\sqrt{2\pi}(1+O(1/\gamma))}{\gamma^{\beta+\frac{k-1}{2}}e^{-\frac{\pi}{2}}\sqrt{2\pi}(1+O(1/\gamma))}\ll T_1^{1-2\beta}.\label{RR2B}
\end{align}
We put $C_3(t)=\sum_{M<m\leq t}\mu_f(m)m^{-i\gamma}$ and $C_4(u)=\sum_{N<n\leq u}\lambda_f(n)n^{-i\gamma}$. Using partial summation formula and Cauchy's inequality, we have
\begin{align}
&\left|\sum_{M<m\leq M'}\frac{\mu_f(m)}{m^\rho}\sum_{N<n\leq N'}\frac{\lambda_f(n)}{n^{1-\rho}}\right|^2\label{RR2C}\\
&\ll\frac{1}{M^{2\beta}}\left(|C_3(M')|^2+\frac{1}{M}\int_M^{M'}|C_3(t)|^2dt\right)\times&\notag\\
&\times\frac{1}{N^{2(1-\beta)}}\left(|C_4(N')|^2+\frac{1}{N}\int_N^{N'}|C_4(u)|^2du\right).\notag& 
\end{align}
From (\ref{RR2A})--(\ref{RR2C}) and the condition $N\leq y$ (i.e. $NT_1^{-1}\ll 1$), we have
\begin{align}
R\ll& T^\varepsilon\sum_{\rho\in\mathcal{E}}T_1^{1-2\beta}M^{-2\beta}N^{2(\beta-1)}|C_3(M_0)|^2|C_4(N_0)|^2\label{RR2M}\\
\ll& T_1^{1-2\sigma+\varepsilon}M^{-2\sigma}N^{-1}\sum_{\rho\in\mathcal{E}}|C_3(N_0)|^2|C_4(M_0)|^2. \label{RRR2}
\end{align}
where $M_0,N_0$ are chosen such that the sum of (\ref{RR2M}) is the maximal value.

We shall calculate the sum of (\ref{RRR2}). Multiplying $C_3(M_0)$ by $C_4(M_0)$, dividing the interval $(MN,M_0N_0]$ into the form $(P_\nu,P_\nu']$ where $P_\nu<P_\nu'\leq 2P_\nu$, $\nu\in\{1,\dots, D_5\}$ (note that $D_5\leq 2$ because $M_0N_0\leq 4MN$), we can calculate the above sum as
\begin{align}
\sum_{\rho\in\mathcal{E}}|C_3(N_0)|^2|C_4(M_0)|^2=\sum_{\rho\in\mathcal{E}}\left|\sum_{\nu=1}^{D_5}\sum_{P_\nu<l\leq P_\nu'}\frac{B(l)}{l^{i\gamma}}\right|^2 \label{XXX1}
\end{align}
where 
\begin{align*}
B(l)=\sum_{\begin{subarray}{c}n|l, \ M<m\leq M_0, \ N<l/m\leq N_0\end{subarray}}\mu_f(m)\lambda_f\left(\frac{l}{m}\right).
\end{align*}
(Note that $ |B(l)|\leq d_4(l)\ll l^\varepsilon$.) 
If we choose $\nu_0$ such that \\ 
$\sum_{\rho\in\mathcal{E}}|\sum_{P_{\nu_0}<l\leq P_{\nu_0}'}B(l)l^{-i\gamma}|$ is the maximal value and apply Lemma \ref{LEM1}, then we have
\begin{align}
\sum_{\rho\in\mathcal{E}}\left|\sum_{\nu=1}^{D_5}\sum_{P_\nu<l\leq P_\nu'}\frac{B(l)}{l^{i\gamma}}\right|^2
\ll\sum_{\rho\in\mathcal{E}}\left|\sum_{P_{\nu_0}<l\leq P_{\nu_0}'}\frac{B(l)}{l^{i\gamma}}\right|^2\ll I_1+\sqrt{I_1I_2} \label{XXX12}
\end{align}
where
\begin{align*}
I_1=\int_{\frac{T_1}{2}}^{T_1}\left|\sum_{P_{\nu_0}<l\leq P_{\nu_0}'}B(l)l^{i\gamma}\right|^2d\gamma, \quad  I_2=\int_{\frac{T_1}{2}}^{T_1}\left|\sum_{P_{\nu_0}<l\leq P_{\nu_0}'}B(l)l^{i\gamma}\log l\right|^2d\gamma.
\end{align*}
Using Lemma \ref{LEM2} we see that 
\begin{align}
I_1\ll MN(T_1+MN)T_1^{\varepsilon}, \quad I_2\ll MN(T_1+MN)T_1^{\varepsilon}. \label{XXX2}
\end{align}
Finally  from (\ref{RRR2})--(\ref{XXX2}) we obtain the desired estimate:
\begin{align*}
R\ll& M^{1-2\sigma}(T_1+MN)T_1^{1-2\sigma+\varepsilon}\notag\\
\ll& (M^{1-2\sigma}T_1^{2(1-\sigma)}+M^{2(1-\sigma)}(NT_1^{-1})T_1^{2-2\sigma})T_1^\varepsilon\notag\\
\ll& (T_1^{2(1-\sigma)}+T_1^{\frac{2}{\delta}(1-\sigma)}T_1^{2(1-\sigma)})T_1^\varepsilon
\ll T_1^{2(1-\sigma)+\varepsilon}
\end{align*}
where $\varepsilon$ is chosen such that $\varepsilon\geq 2/\delta$.
\end{proof}
By Propositions \ref{PROEA}--\ref{PROEC}, the proof of Theorem \ref{ZZZ1} is completed.





\begin{thebibliography}{30}
\bibitem{APO} T. M. Apostol, \textit{Modular functions and Dirichlet series in number theory}, second edition, Springer-Verlag, 1990.
\bibitem{B&L} H. Bohr and E. Landau, \textit{Sur les z\'{e}ros de la fonction $\zeta(s)$ de Riemann}, Comp. Rend. Acad. Sci. Paris \textbf{158} (1914), 106--110.
\bibitem{DEL} P. Deligne, \textit{La conjecture de Weil. I}, Publ. Math. Inst. Hautes \'{E}tudes Sci. \textbf{43} (1974), 273--307.
\bibitem{GD1} A. Good, \textit{Approximative Funktionalgleichungen und Mittelwerts\"{a}tze f\"{u}r Dirichletreihen, die Spitzenformen assoziiert sind}, Comment. Math. Helv. \textbf{50} (1975), 327--361.
\bibitem{GD2} A. Good, \textit{The square mean of Dirichlet series associated with cusp forms},  Mathematika, \textbf{29} (1982), 278--295.
\bibitem{HL3} G. H. Hardy and J. E. Littlewood, \textit{The approximate functional equation for $\zeta(s)$ and $\zeta(s)^2$}, Proc. London Math. Soc. (2) \textbf{29} (1929), 81--97.
\bibitem{HEC} E. Hecke, \textit{\"{U}ber Modulfunktionen und die Dirichletschen Reihen mit Eulerscher Produktentwicklung. I}, Math. Ann. \textbf{114} (1937), 1--28.
\bibitem{HUX} M. N. Huxley, \textit{On the difference between consecutive primes}, Invent. math. \textbf{15} (1972), 164--170.
\bibitem{IN1} A. Ingham, \textit{On the estimation of $N(\sigma,T)$}, Quart. J. Math. Oxford \textbf{11} (1940), 291--292.
\bibitem{IVI} A. Ivi\'{c}, \textit{On zeta-functions associated with Fourier coefficients of cusp forms}, in Proceedings of the Amalfi Conference on Analytic Number Theory, E. Bombieri et al. (eds.), Universit\`{a} di Salerno, 1992, 231--246.
\bibitem{JUT} M. Jutila, \textit{Lectures on a Method in the Theory of Exponential Sums}, Tata Inst. Fund. Res. Lectures on Math. and Phys. \textbf{80}, Springer, Berlin, 1987.
\bibitem{KAR0} A. A. Karatsuba, \textit{The distribution of prime numbers}, Russ. Math. Surv. \textbf{45} (1990), 99--171.
\bibitem{KAR} A. A. Karatsuba and S. M. Voronin, \textit{The Riemann Zeta-Function},  Walter de Gruyter de Gruyter Explosions in Mathematics 5, 1992.
\bibitem{MOR} C. J. Moreno, \textit{``Explicit formulas in the theory of automorphic forms" in \textit{Number Theory Day}}, Lecture Notes in Math. \textbf{626}, Springer-Verlag, Berlin (1977), 73--216.
\end{thebibliography}
\end{document}